\documentclass[final,hidelinks]{siamart0216}

\usepackage{amsfonts}
\usepackage{mathrsfs}
\usepackage{graphicx}
\usepackage{epstopdf}
\usepackage[algo2e,ruled,linesnumbered,noline,algosection]{algorithm2e}
\usepackage{multirow}
\usepackage{comment}
\usepackage{color}
\usepackage{tabto}
\ifpdf
  \DeclareGraphicsExtensions{.eps,.pdf,.png,.jpg}
\else
  \DeclareGraphicsExtensions{.eps}
\fi

\newsiamremark{remark}{Remark}

\crefname{algocf}{Algorithm}{Algorithms}
\Crefname{algocf}{Algorithm}{Algorithms}

\renewcommand{\theequation}{\thesection.\arabic{equation}}
\renewcommand{\thefigure}{\thesection.\arabic{figure}}
\renewcommand{\thetable}{\thesection.\arabic{table}}

\numberwithin{theorem}{section}
\numberwithin{equation}{section}
\numberwithin{figure}{section}
\numberwithin{table}{section}
\numberwithin{algorithm}{section}

\title{A Low-Rank Multigrid Method for the Stochastic Steady-State Diffusion Problem%
   \thanks{This work was supported by the U.S. Department of Energy Office of Advanced Scientific Computing Research, Applied Mathematics program under award DE-SC0009301 and by the U.S. National Science Foundation under grant DMS1418754.}}

\author{
   Howard C. Elman%
   \thanks{Department of Computer Science and Institute for Advanced Computer Studies, University of Maryland, College Park, MD 20742 (\email{elman@cs.umd.edu}).}
   \and
   Tengfei Su%
   \thanks{Applied Mathematics \& Statistics, and Scientific Computation Program, University of Maryland, College Park, MD 20742 (\email{tengfesu@math.umd.edu}).}
}

\headers{Low-Rank Multigrid for Stochastic Diffusion Problem}
{H. C. Elman and T. Su}

\begin{document}

\maketitle

\begin{abstract}
We study a multigrid method for solving large linear systems of equations with tensor product structure. Such systems are obtained from stochastic finite element discretization of stochastic partial differential equations such as the steady-state diffusion problem with random coefficients. When the variance in the problem is not too large, the solution can be well approximated by a low-rank object. In the proposed multigrid algorithm, the matrix iterates are truncated to low rank to reduce memory requirements and computational effort. The method is proved convergent with an analytic error bound. Numerical experiments show its effectiveness in solving the Galerkin systems compared to the original multigrid solver, especially when the number of degrees of freedom associated with the spatial discretization is large.
\end{abstract}



\section{Introduction}
Stochastic partial differential equations (SPDEs) arise from physical applications where the parameters of the problem are subject to uncertainty. Discretization of SPDEs gives rise to large linear systems of equations which are computationally expensive to solve. These systems are in general sparse and structured. In particular, the coefficient matrix can often be expressed as a sum of tensor products of smaller matrices \cite{OlUl10,MaZa12,PoEl09}. For such systems it is natural to use an iterative solver where the coefficient matrix is never explicitly formed and matrix-vector products are computed efficiently. One way to further reduce costs is to construct low-rank approximations to the desired solution. The iterates are truncated so that the solution method handles only low-rank objects in each iteration. This idea has been used to reduce the costs of iterative solution algorithms based on Krylov subspaces. For example, a low-rank conjugate gradient method was given in \cite{KrTo11}, and low-rank generalized minimal residual methods have been studied in \cite{BaGr13,LeEl16}. 

In this study, we propose a low-rank multigrid method for solving the Galerkin systems. We consider a steady-state diffusion equation with random diffusion coefficient as model problem, and we use the stochastic finite element method (SFEM, see \cite{BaTe04,GhSp03}) for the discretization of the problem.  The resulting Galerkin system has tensor product structure and moreover, quantities used in the computation, such as the solution sought, can be expressed in matrix format. It has been shown that such systems admit low-rank approximate solutions \cite{BeOn15,KrTo11}. In our proposed multigrid solver, the matrix iterates are truncated to have low rank in each iteration. We derive an analytic bound for the error of the solution and show the convergence of the algorithm. We demonstrate using benchmark problems that the low-rank multigrid solver is often more efficient than a solver that does not use truncation, and that it is especially advantageous in reducing computing time for large-scale problems.

An outline of the paper is as follows. In \cref{sec:model} we state the problem and briefly review the stochastic finite element method and the multigrid solver for the stochastic Galerkin system from which the new technique is derived. In \cref{sec:low-rank} we discuss the idea of low-rank approximation and introduce the multigrid solver with low-rank truncation. A convergence analysis of the low-rank multigrid solver is also given in this section. The results of numerical experiments are shown in \cref{sec:numerical} to test the performance of the algorithm, and some conclusions are drawn in the last section.

\section{Model problem}
\label{sec:model}
Consider the stochastic steady-state diffusion equation with homogeneous Dirichlet boundary conditions
\begin{equation} \label{eq:stoch_diff}
\begin{cases}
-\nabla\cdot(c(x,\omega)\nabla u(x,\omega)) = f(x) & \text{in } D\times\Omega,\\
u(x,\omega) = 0 & \text{on } \partial D\times\Omega.\\
\end{cases}
\end{equation}
Here $D$ is a spatial domain and $\Omega$ is a sample space with $\sigma$-algebra $\mathscr{F}$ and probability measure $P$. The diffusion coefficient $c(x,\omega): D\times\Omega\rightarrow\mathbb{R}$ is a random field. We consider the case where the source term $f$ is deterministic. The stochastic Galerkin formulation of \cref{eq:stoch_diff} uses a weak formulation: find $u(x,\omega)\in\mathbb{V}=H_0^1(D)\otimes L^2(\Omega)$ satisfying
\begin{equation} \label{eq:weak1}
\int_\Omega\int_D c(x,\omega)\nabla u(x,\omega)\cdot\nabla v(x,\omega)\text{d}x\text{d}P = \int_\Omega\int_D f(x)v(x,\omega)\text{d}x\text{d}P
\end{equation}
for all $v(x,\omega)\in\mathbb{V}$. The problem is well posed if $c(x,\omega)$ is bounded and strictly positive, i.e.,
\begin{displaymath}
0<c_1\leq c(x,\omega)\leq c_2<\infty,\,\text{a.e. } \forall x\in D,
\end{displaymath}
so that the Lax-Milgram lemma establishes existence and uniqueness of the weak solution. 

We will assume that the stochastic coefficient $c(x,\omega)$ is represented as a truncated Karhunen-Lo$\grave{\text{e}}$ve (KL) expansion \cite{Loeve,Lord}, in terms of a finite collection of uncorrelated random variables $\{\xi_l\}_{l=1}^m$:
\begin{equation} \label{eq:kl}
c(x,\omega) \approx c_0(x) + \sum_{l=1}^m \sqrt{\lambda_l}c_l(x)\xi_l(\omega)
\end{equation}
where $c_0(x)$ is the mean function, $(\lambda_l,c_l(x))$ is the $l$th eigenpair of the covariance function $r(x,y)$, and the eigenvalues $\{\lambda_l\}$ are assumed to be in non-increasing order. In \cref{sec:numerical} we will further assume these random variables are independent and identically distributed. Let $\rho(\xi)$ be the joint density function and $\Gamma$ be the joint image of $\{\xi_l\}_{l=1}^m$. The weak form of \cref{eq:stoch_diff} is then given as follows: find $u(x,\xi)\in\mathbb{W}=H_0^1(D)\otimes L^2(\Gamma)$ s.t.
\begin{equation} \label{eq:weak2}
\int_\Gamma \rho(\xi)\int_D c(x,\xi)\nabla u(x,\xi)\cdot\nabla v(x,\xi)\text{d}x\text{d}\xi = \int_\Gamma \rho(\xi)\int_D f(x)v(x,\xi)\text{d}x\text{d}\xi
\end{equation}
for all $v(x,\xi)\in\mathbb{W}$.

\subsection{Stochastic finite element method}
We briefly review the stochastic finite element method as described in \cite{BaTe04,GhSp03}. This method approximates the weak solution of \cref{eq:stoch_diff} in a finite-dimensional subspace
\begin{equation}
\mathbb{W}^{hp} = S^h\otimes T^p=\text{span}\{\phi(x)\psi(\xi) \mid \phi(x)\in S^h,\psi(\xi)\in T^p\},
\end{equation}
where $S^h$ and $T^p$ are finite-dimensional subspaces of $H_0^1(D)$ and $L^2(\Gamma)$. We will use quadrilateral elements and piecewise bilinear basis functions $\{\phi(x)\}$ for the discretization of the physical space $H_0^1(D)$, and generalized polynomial chaos \cite{XiKa03} for the stochastic basis functions $\{\psi(\xi)\}$. The latter are $m$-dimensional orthogonal polynomials whose total degree doesn't exceed $p$. The orthogonality relation means
\[
\int_\Gamma\psi_r(\xi)\psi_s(\xi)\rho(\xi)\text{d}\xi=\delta_{rs}\int_\Gamma \psi_r^2(\xi)\rho(\xi)\text{d}\xi.
\]
For instance, Legendre polynomials are used if the random variables have uniform distribution with zero mean and unit variance. The number of degrees of freedom in $T^p$ is
\[
N_{\xi}=\frac{(m+p)!}{m!p!}.
\]

Given the subspace, now one can write the SFEM solution as a linear combination of the basis functions,
\begin{equation} \label{eq:sfem_soln}
u_{hp}(x,\xi)=\sum_{j=1}^{N_x}\sum_{s=1}^{N_{\xi}} u_{js}\phi_j(x)\psi_s(\xi),
\end{equation}
where $N_x$ is the dimension of the subspace $S^h$. Substituting \cref{eq:kl,eq:sfem_soln} into \cref{eq:weak2}, and taking the test function as any basis function $\phi_i(x)\psi_r(\xi)$ results in the Galerkin system: find $\mathbf{u}\in\mathbb{R}^{N_xN_{\xi}}$, s.t.
\begin{equation} \label{eq:galerkin}
A\mathbf{u}=\mathbf{f}.
\end{equation}
The coefficient matrix $A$ can be represented in tensor product notation \cite{PoEl09},
\begin{equation}
A=G_0\otimes K_0+\sum_{l=1}^m G_l\otimes K_l,
\end{equation}
where $\{K_l\}_{l=0}^m$ are the stiffness matrices and $\{G_l\}_{l=0}^m$ correspond to the stochastic part, with entries
\begin{equation}
\begin{aligned}
G_0(r,s)&=\int_\Gamma\psi_r(\xi)\psi_s(\xi)\rho(\xi)\text{d}\xi,\,K_0(i,j)= \int_D c_0(x)\nabla\phi_i(x)\nabla\phi_j(x)\text{d}x,\\
G_l(r,s)&=\int_\Gamma\xi_l\psi_r(\xi)\psi_s(\xi)\rho(\xi)\text{d}\xi,\,K_l(i,j)= \int_D \sqrt{\lambda_l}c_l(x)\nabla\phi_i(x)\nabla\phi_j(x)\text{d}x,
\end{aligned}
\end{equation}
$l=1,\dots,m;\,r,s=1,\ldots,N_\xi;\,i,j=1,\ldots,N_x$. The right-hand side can be written as a tensor product of two vectors:
\begin{equation}
\mathbf{f}=g_0\otimes f_0,
\end{equation}
where
\begin{equation}
\begin{aligned}
g_0(r)&=\int_\Gamma\psi_r(\xi)\rho(\xi)\text{d}\xi,\,\, r=1,\ldots,N_\xi,\\
f_0(i)&=\int_D f(x)\phi_i(x)\text{d}x,\,\, i=1,\ldots,N_x.
\end{aligned}
\end{equation}

Note that in the Galerkin system \cref{eq:galerkin}, the matrix $A$ is symmetric and positive definite. It is also blockwise sparse (see \cref{fig:sparsity}) due to the orthogonality of $\{\psi_r(\xi)\}$. The size of the linear system is in general very large ($N_xN_\xi\times N_xN_\xi$). For such a system it is suitable to use an iterative solver. Multigrid methods are among the most effective iterative solvers for the solution of discretized elliptic PDEs, capable of achieving convergence rates that are independent of the mesh size, with computational work growing only linearly with the problem size \cite{Hack85,Saad03}. 
\begin{figure}[tbhp]
\includegraphics[width=0.32\textwidth]{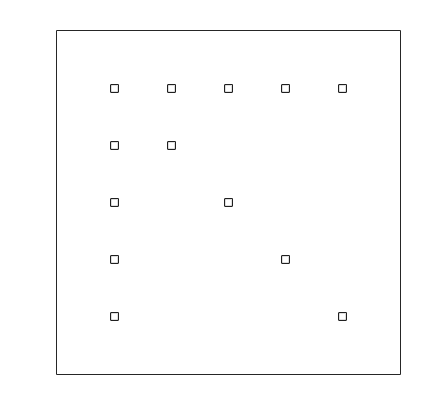}
\includegraphics[width=0.32\textwidth]{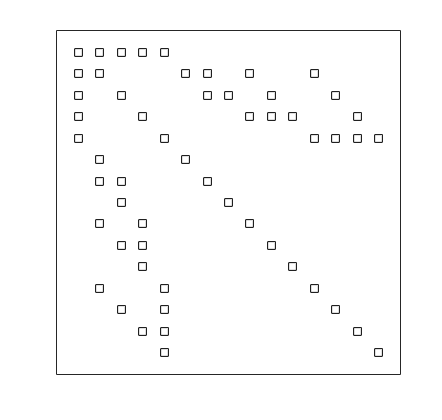}
\includegraphics[width=0.32\textwidth]{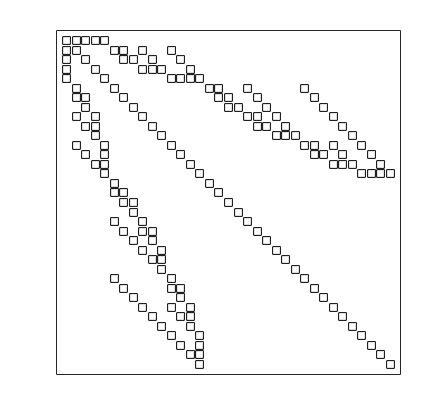}
\caption{Block structure of $A$. $m=4,p=1,2,3$ from left to right. Block size is $N_x\times N_x$.}
\label{fig:sparsity}
\end{figure}

\subsection{Multigrid}
\label{sec:multigrid}
In this subsection we discuss a geometric multigrid solver proposed in \cite{ElFu07} for the solution of the stochastic Galerkin system \cref{eq:galerkin}. For this method, the mesh size $h$ varies for different grid levels, while the polynomial degree $p$ is held constant, i.e., the fine grid space and coarse grid space are defined as
\begin{equation}
\mathbb{W}^{hp}=S^h\otimes T^p,\quad \mathbb{W}^{2h,p}=S^{2h}\otimes T^p,
\end{equation}
respectively. Then the prolongation and restriction operators are of the form
\begin{equation} \label{eq:grid_transfer}
\mathscr{P}=I\otimes P,\quad\mathscr{R}=I\otimes P^T,
\end{equation}
where $P$ is the same prolongation matrix as in the deterministic case. On the coarse grid we only need to construct matrices $\{K^{2h}_l\}_{l=0}^m$, and 
\begin{equation}
{A}^{2h}=G_0\otimes {K}^{2h}_0+\sum_{l=1}^m G_l\otimes {K}^{2h}_l.
\end{equation}
The matrices $\{G_l\}_{l=0}^m$ are the same for all grid levels.

\Cref{alg:mg} describes the complete multigrid method. In each iteration, we apply one multigrid cycle (\textsc{Vcycle}) for the residual equation
\begin{equation}
A\mathbf{c}^{(i)}=\mathbf{r}^{(i)}=\mathbf{f}-A\mathbf{u}^{(i)}
\end{equation}
and update the solution $\mathbf{u}^{(i)}$ and residual $\mathbf{r}^{(i)}$. The \textsc{Vcycle} function is called recursively. On the coarsest grid level ($h=h_0$) we form matrix $A$ and solve the linear system directly. The system is of order $O(N_\xi)$ since $A\in\mathbb{R}^{N_xN_\xi\times N_xN_\xi}$ where $N_x$ is a very small number on the coarsest grid. The smoothing function (\textsc{Smooth}) is based on a matrix splitting $A=Q-Z$ and stationary iteration
\begin{equation} \label{eq:smooth}
\mathbf{u}_{s+1}  = \mathbf{u}_s+Q^{-1}(\mathbf{f}-A\mathbf{u}_s),
\end{equation}
which we assume is convergent, i.e., the spectral radius $\rho(I-Q^{-1}A)<1.$ The algorithm is run until the specified relative tolerance $tol$ or maximum number of iterations $maxit$ is reached. It is shown in \cite{ElFu07} that for $f\in L^2(D)$, the convergence rate of this algorithm is independent of the mesh size $h$, the number of random variables $m$, and the polynomial degree $p$.

\begin{algorithm2e}
\caption{Multigrid for stochastic Galerkin systems}
\label{alg:mg}
\SetFuncSty{textsc}
\SetNlSty{}{}{:}
\DontPrintSemicolon
\SetKw{Init}{initialization}
\SetKwProg{Func}{function}{}{end}
\SetKwFunction{Vcycle}{Vcycle}
\SetKwFunction{Smooth}{Smooth}
\Init: $i=0$, $\mathbf{r}^{(0)}=\mathbf{f}$, $r_0=\lVert\mathbf{f}\rVert_2$ \;
\While{$r>tol*r_0$ $\&$ $i\leq maxit$}{
	$\mathbf{c}^{(i)}=$ \Vcycle{$A,\mathbf{0},\mathbf{r}^{(i)}$} \;
	$\mathbf{u}^{(i+1)}=\mathbf{u}^{(i)}+\mathbf{c}^{(i)}$ \;
	$\mathbf{r}^{(i+1)} = \mathbf{f}-A\mathbf{u}^{(i+1)}$ \;
	$r=\lVert \mathbf{r}^{(i+1)}\rVert_2$, $i=i+1$ \;
}
\BlankLine
\Func{$\mathbf{u}^h=$ \Vcycle{$A^h,\mathbf{u}^h_0,\mathbf{f}^h$}}{
	\uIf{$h==h_0$}{
		solve $A^h\mathbf{u}^{h}=\mathbf{f}^h$ directly \;
	}
	\Else{
		$\mathbf{u}^h=$ \Smooth{$A^h,\mathbf{u}^h_0,\mathbf{f}^h$} \;
		$\mathbf{r}^h=\mathbf{f}^h-A^h\mathbf{u}^h$ \;
		$\mathbf{r}^{2h}=\mathscr{R}\mathbf{r}^h$ \;
		$\mathbf{c}^{2h}=$ \Vcycle{$A^{2h},\mathbf{0},\mathbf{r}^{2h}$} \;
		$\mathbf{u}^{h}=\mathbf{u}^{h}+\mathscr{P}{\mathbf{c}^{2h}}$ \;
		$\mathbf{u}^h=$ \Smooth{$A^h,\mathbf{u}^h,\mathbf{f}^h$} \;
	}
}
\BlankLine
\Func{$\mathbf{u}=$ \Smooth{$A,\mathbf{u},\mathbf{f}$}}{
	\For{$\nu$ steps}{
		$\mathbf{u} = \mathbf{u}+Q^{-1}(\mathbf{f}-A\mathbf{u})$ \;
	}
}
\end{algorithm2e}

\section{Low-rank approximation}
\label{sec:low-rank}
In this section we consider a technique designed to reduce computational effort, in terms of both time and memory use, using low-rank methods. We begin with the observation that the solution vector of the Galerkin system \cref{eq:galerkin}
\begin{displaymath}
\mathbf{u}=[u_{11},u_{21},\ldots,u_{N_x1},\ldots,u_{1N_\xi},u_{2N_\xi},\ldots,u_{N_xN_\xi}]^T \in\mathbb{R}^{N_xN_\xi}
\end{displaymath}
can be restructured as a matrix
\begin{equation}
U = \text{mat}(\mathbf{u})=
  \begin{pmatrix}
  u_{11} & u_{12} & \cdots & u_{1N_\xi} \\
  u_{21} & u_{22} & \cdots & u_{2N_\xi} \\
  \vdots  & \vdots  & \ddots & \vdots  \\
  u_{N_x1} & u_{N_x2} & \cdots & u_{N_xN_\xi} 
  \end{pmatrix}\in\mathbb{R}^{N_x\times N_\xi}.
\end{equation} 
Then (\ref{eq:galerkin}) is equivalent to a system in matrix format,
\begin{equation} \label{eq:matrix}
\mathcal{A}(U)=F,
\end{equation}
where
\begin{equation}
\begin{aligned}
&\mathcal{A}(U)=K_0UG_0^T+\sum_{l=1}^m K_lUG_l^T,\\
&F=\text{mat}(\mathbf{f})=\text{mat}(g_0\otimes f_0)=f_0 g_0^T.
\end{aligned}
\end{equation}
It has been shown in \cite{BeOn15,KrTo11} that the ``matricized'' version of the solution $U$ can be well approximated by a low-rank matrix when $N_xN_\xi$ is large. Evidence of this can be seen in \cref{fig:decay1}, which shows the singular values of the exact solution $U$ for the benchmark problem discussed in \cref{sec:numerical}. In particular, the singular values decay exponentially, and low-rank approximate solutions can be obtained by dropping terms from the singular value decomposition corresponding to small singular values.
\begin{figure}[tbhp]
    \includegraphics[width=0.4\textwidth]{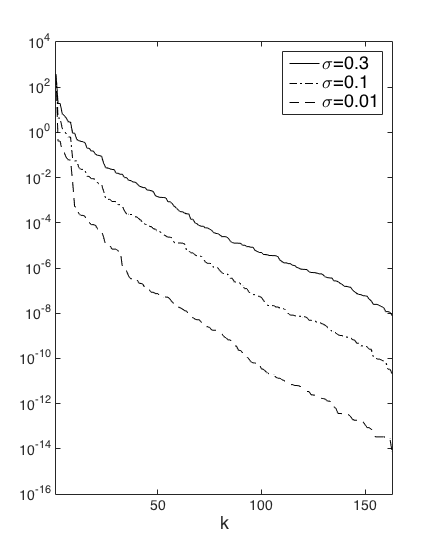}
    \includegraphics[width=0.4\textwidth]{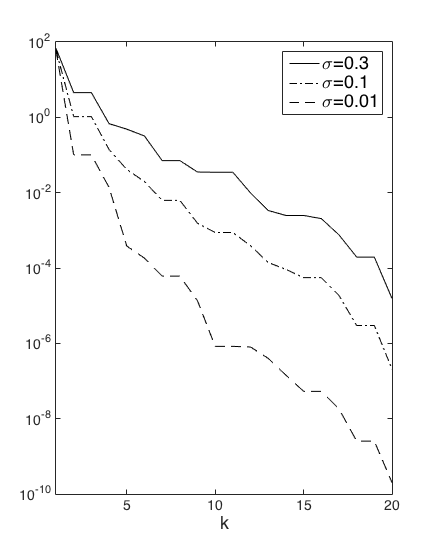}
    \centering
\caption{Decay of singular values of solution matrix $U$. Left: exponential covariance, $b=5$, $h=2^{-6}$, $m=8$, $p=3$. Right: squared exponential covariance, $b=2$, $h=2^{-6}$, $m=3$, $p=3$. See the benchmark problem in \cref{sec:numerical}.}
\label{fig:decay1}
\end{figure}

Now we use low-rank approximation in the multigrid solver for \cref{eq:matrix}. Let $U^{(i)}=\text{mat}(\mathbf{u}^{(i)})$ be the $i$th iterate, expressed in matricized format\footnote{In the sequel, we use $\mathbf{u}^{(i)}$ and $U^{(i)}$ interchangeably to represent the equivalent vectorized or matricized quantities.}, and suppose $U^{(i)}$ is represented as the outer product of two rank-$k$ matrices, i.e., $U^{(i)}\approx V^{(i)}W^{(i)T}$, where $V^{(i)}\in\mathbb{R}^{N_x\times k}$, $W^{(i)}\in\mathbb{R}^{N_\xi\times k}$. This factored form is convenient for implementation and can be readily used in basic matrix operations. For instance, the sum of two matrices gives
\begin{equation}
V_1^{(i)}W_1^{(i)T}+V_2^{(i)}W_2^{(i)T}=[V_1^{(i)},V_2^{(i)}][W_1^{(i)},W_2^{(i)}]^T.
\end{equation}
Similarly, $\mathcal{A}(V^{(i)}W^{(i)T})$ can also be written as an outer product of two matrices:
\begin{equation} \label{eq:opt_a}
\begin{aligned}
\mathcal{A}(V^{(i)}W^{(i)T}) &=(K_0V^{(i)})(G_0W^{(i)})^T+\sum_{l=1}^m (K_lV^{(i)})(G_lW^{(i)})^T\\
=& \,\,[K_0V^{(i)},K_1V^{(i)},\ldots,K_mV^{(i)}][G_0W^{(i)},G_1W^{(i)},\ldots,G_mW^{(i)}]^T.\\
\end{aligned}
\end{equation}
If $V^{(i)},W^{(i)}$ are used to represent iterates in the multigrid solver and $k\ll \text{min}(N_x,N_\xi)$, then both memory and computational (matrix-vector products) costs can be reduced, from $O(N_xN_\xi)$ to $O((N_x+N_\xi)k)$. Note, however, that the ranks of the iterates may grow due to matrix additions. For example, in \cref{eq:opt_a} the rank may increase from $k$ to $(m+1)k$ in the worst case. A way to prevent this from happening, and also to keep costs low, is to truncate the iterates and force their ranks to remain low.

\subsection{Low-rank truncation}
Our truncation strategy is derived using an idea from \cite{KrTo11}. Assume ${\tilde X}=\tilde{V}\tilde{W}^T$, $\tilde{V}\in\mathbb{R}^{N_x\times \tilde{k}}$, $\tilde{W}\in\mathbb{R}^{N_\xi\times \tilde{k}}$, and ${X}=\mathcal{T}(\tilde X)$ is truncated to rank $k$ with $X=VW^T$, $V\in\mathbb{R}^{N_x\times k}$, $W\in\mathbb{R}^{N_\xi\times k}$ and $k<\tilde k$. First, compute the QR factorization for both $\tilde V$ and $\tilde W$,
\begin{equation}
\tilde V=Q_{\tilde V}R_{\tilde V}, \quad \tilde W=Q_{\tilde W}R_{\tilde W},\quad \text{ so } {\tilde X}=Q_{\tilde V}R_{\tilde V}R_{\tilde W}^TQ_{\tilde W}^T.
\end{equation}
The matrices $R_{\tilde V}$ and $R_{\tilde W}$ are of size ${\tilde k}\times {\tilde k}$. Next, compute a singular value decomposition (SVD) of the small matrix $R_{\tilde V}R_{\tilde W}^T$:
\begin{equation}
R_{\tilde V}R_{\tilde W}^T=\hat V \text{diag}(\sigma_1,\ldots,\sigma_{\tilde k})\hat W^T
\end{equation}
where $\sigma_1,\ldots,\sigma_{\tilde k}$ are the singular values in descending order. We can truncate to a rank-$k$ matrix where $k$ is specified using either a relative criterion for singular values,
\begin{equation}
\label{eq:rel_tol}
\sqrt{\sigma_{k+1}^2+\cdots+\sigma_{{\tilde k}}^2}\leq \epsilon_{\text{rel}}\sqrt{\sigma_{1}^2+\cdots+\sigma_{{\tilde k}}^2}
\end{equation}
or an absolute one,
\begin{equation} \label{eq:abs_tol}
k =\text{max}\{k\mid\sigma_k\geq\epsilon_{\text{abs}}\}.
\end{equation}
Then the truncated matrices can be written in MATLAB notation as
\begin{equation}
V=Q_{\tilde V}\hat V(:,1:k),\quad W=Q_{\tilde W}\hat W(:,1:k)\text{diag}(\sigma_1,\ldots,\sigma_{k}).
\end{equation}
Note that the low-rank matrices ${X}$ obtained from \cref{eq:rel_tol} and \cref{eq:abs_tol} satisfy
\begin{equation} \label{eq:rel_tol_equiv}
\lVert X-\tilde{X} \rVert_F \leq \epsilon_\text{rel}\lVert {\tilde X}\rVert_F
\end{equation}
and
\begin{equation} \label{eq:abs_tol_equiv}
\lVert X-\tilde{X} \rVert_F \leq \epsilon_\text{abs}\sqrt{\tilde k-k},
\end{equation}
respectively. The right-hand side of \cref{eq:abs_tol_equiv} is bounded by $\sqrt{N_\xi}\epsilon_\text{abs}$ since in general $N_\xi<N_x$. The total cost of this computation is $O((N_x+N_\xi+{\tilde k}){\tilde k}^2)$. In the case where $\tilde k$ becomes larger than $N_\xi$, we compute instead a direct SVD for $\tilde X$, which requires a matrix-matrix product to compute $\tilde X$ and an SVD, with smaller total cost $O(N_x N_\xi\tilde k + N_xN_\xi^2)$.

\subsection{Low-rank multigrid}
\label{sec:lrmg}
The multigrid solver with low-rank truncation is given in \cref{alg:low-rank}. It uses truncation operators $\mathcal{T}_\text{rel}$ and $\mathcal{T}_\text{abs}$, which are defined using a relative and an absolute criterion, respectively. In each iteration, one multigrid cycle (\textsc{Vcycle}) is applied to the residual equation. Since the overall magnitudes of the singular values of the correction matrix $C^{(i)}$ decrease as $U^{(i)}$ converges to the exact solution (see \cref{fig:decay2} for example), it is suitable to use a relative truncation tolerance $\epsilon_\text{rel}$ inside the \textsc{Vcycle} function. In the smoothing function (\textsc{Smooth}), the iterate is truncated after each smoothing step using a relative criterion
\begin{equation} \label{eq:rel_tol1}
\lVert \mathcal{T}_\text{rel$_1$} (U) - U \rVert_F \leq \epsilon_\text{rel} \lVert F^h-\mathcal{A}^h(U^h_0) \rVert_F
\end{equation}
where $A^h$, $U^h_0$, and $F^h$ are arguments of the \textsc{Vcycle} function, and $F^h-\mathcal{A}^h(U^h_0)$ is the residual at the beginning of each V-cycle. In Line \ref{lst:line:13}, the residual is truncated via a more stringent relative criterion
\begin{equation} \label{eq:rel_tol2}
\lVert \mathcal{T}_\text{rel$_2$} (R^h) - R^h \rVert_F \leq \epsilon_\text{rel}h \lVert F^h-\mathcal{A}^h(U^h_0) \rVert_F
\end{equation}
where $h$ is the mesh size. In the main \textbf{while} loop, an absolute truncation criterion \cref{eq:abs_tol} with tolerance $\epsilon_\text{abs}$ is used and all the singular values of $U^{(i)}$ below $\epsilon_\text{abs}$ are dropped. The algorithm is terminated either when the largest singular value of the residual matrix $R^{(i)}$ is smaller than $\epsilon_\text{abs}$ or when the multigrid solution reaches the specified accuracy.

\begin{algorithm2e}
\caption{Multigrid with low-rank truncation}
\label{alg:low-rank}
\SetFuncSty{textsc}
\SetNlSty{}{}{:}
\DontPrintSemicolon
\SetKw{Init}{initialization}
\SetKwProg{Func}{function}{}{end}
\SetKwFunction{Vcycle}{Vcycle}
\SetKwFunction{Smooth}{Smooth}
\Init: $i=0$, $R^{(0)}=F$ in low-rank format, $r_0=\lVert F\rVert_F$ \;
\While{$r>tol*r_0$ $\&$ $i\leq maxit$}{
	$C^{(i)}=$ \Vcycle{$A,0,R^{(i)}$} \;
	$\tilde U^{(i+1)}=U^{(i)}+C^{(i)},$\tabto{6cm}$U^{(i+1)}=\mathcal{T}_{\text{abs}}(\tilde U^{(i+1)})$ \;
	$\tilde R^{(i+1)} = F-\mathcal{A}(U^{(i+1)}),$\tabto{6cm}$R^{(i+1)}=\mathcal{T}_{\text{abs}}(\tilde R^{(i+1)})$ \;
	$r=\lVert R^{(i+1)}\rVert_F$, $i=i+1$ \;
}
\BlankLine
\Func{$U^h=$ \Vcycle{$A^h,U_0^h,F^h$}}{
	\uIf{$h==h_0$}{
		solve $\mathcal{A}^h(U^h)=F^h$ directly \;
	}
	\Else{
		$U^h=$ \Smooth{$A^h,U_0^h,F^h$} \;
		$\tilde R^h=F^h-\mathcal{A}^h(U^h),$\tabto{5.5cm}$R^h=\mathcal{T}_\text{rel$_2$}(\tilde R^h)$ \label{lst:line:13} \;
		${R}^{2h}=\mathscr{R}(R^h)$ \;
		${C}^{2h}=$ \Vcycle{${A}^{2h},0,R^{2h}$} \;
		$U^{h}=U^h+\mathscr{P}(C^{2h})$ \;
		$U^h=$ \Smooth{$A^h,U^h,F^h$} \;
	}
}
\BlankLine
\Func{$U=$ \Smooth{$A,U,F$}}{
	\For{$\nu$ steps}{
		$\tilde U = U+\mathscr{S}(F-\mathcal{A}(U))$, \tabto{5.5cm}$U=\mathcal{T}_{\text{rel$_1$}}(\tilde U)$ \;
	}
}
\end{algorithm2e}

\begin{figure}[tbhp]
\includegraphics[width=0.75\textwidth]{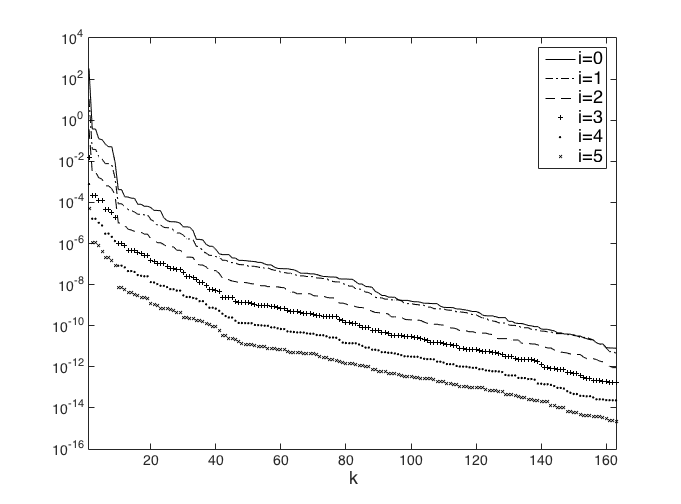}
\centering
\caption{Singular values of the correction matrix $C^{(i)}$ at multigrid iteration $i=0,1,\ldots,5$ (without truncation) for the benchmark problem in \cref{sec:numerical1} (with $\sigma=0.01$, $b=5$, $h=2^{-6}$, $m=8$, $p=3$).}
\label{fig:decay2}
\end{figure}

Note that the post-smoothing is not explicitly required in \cref{alg:mg,alg:low-rank}, and we include it just for sake of completeness. Also, in \cref{alg:low-rank}, if the smoothing operator has the form $\mathscr{S}=S_1\otimes S_2$, then for any matrix with a low-rank factorization $X=VW^T$, application of the smoothing operator gives
\begin{equation} \label{eq:lr_smooth}
\mathscr{S}(X)=\mathscr{S}(VW^T)=(S_2V)(S_1W)^T,
\end{equation} 
so that the result is again the outer product of two matrices of the same low rank. The prolongation and restriction operators \cref{eq:grid_transfer} are implemented in a similar manner. Thus, the smoothing and grid-transfer operators do not affect the ranks of matricized quantities in \cref{alg:low-rank}.

\subsection{Convergence analysis}
\label{sec:convergence}
In order to show that \cref{alg:low-rank} is convergent, we need to know how truncation affects the contraction of error. Consider the case of a two-grid algorithm for the linear system $A\mathbf{u}=\mathbf{f}$, where the coarse-grid solve is exact and no post-smoothing is done. Let $\bar{A}$ be the coefficient matrix on the coarse grid, let $\mathbf{e}^{(i)}=\mathbf{u}-\mathbf{u}^{(i)}$ be the error associated with $\mathbf{u}^{(i)}$, and let $\mathbf{r}^{(i)}=\mathbf{f}-A\mathbf{u}^{(i)}=A\mathbf{e}^{(i)}$ be the residual. It is shown in \cite{ElFu07} that if no truncation is done, the error after a two-grid cycle becomes
\begin{equation}
\label{eq:error_exact}
\mathbf{e}_\text{notrunc}^{(i+1)}=(A^{-1}-\mathscr{P}\bar{A}^{-1}\mathscr{R})A(I-Q^{-1}A)^\nu \mathbf{e}^{(i)},
\end{equation}
and
\begin{equation}
\label{eq:bound_exact}
\lVert \mathbf{e}_\text{notrunc}^{(i+1)}\rVert_A \leq C\eta(\nu)\lVert \mathbf{e}^{(i)}\rVert_A,
\end{equation}
where $\nu$ is the number of pre-smoothing steps, $C$ is a constant, and $\eta(\nu)\rightarrow 0$ as $\nu\rightarrow\infty$. The proof consists of establishing the smoothing property
\begin{equation} \label{eq:smooth_property}
\lVert A(I-Q^{-1}A)^\nu\mathbf{y} \rVert_2 \leq \eta(\nu)\lVert \mathbf{y}\rVert_A,\quad \forall \mathbf{y}\in\mathbb{R}^{N_xN_\xi},
\end{equation}
and the approximation property
\begin{equation} \label{eq:approx_property}
\lVert (A^{-1}-\mathscr{P}\bar{A}^{-1}\mathscr{R})\mathbf{y} \rVert_A \leq C\lVert \mathbf{y}\rVert_2,\quad \forall \mathbf{y}\in\mathbb{R}^{N_xN_\xi},
\end{equation}
and applying these bounds to \cref{eq:error_exact}.

Now we derive an error bound for \cref{alg:low-rank}. The result is presented in two steps. First, we consider the \textsc{Vcycle} function only; the following lemma shows the effect of the relative truncations defined in \cref{eq:rel_tol1,eq:rel_tol2}.

\begin{lemma}
Let $\mathbf{u}^{(i+1)}=\textsc{Vcycle}(A,\mathbf{u}^{(i)},\mathbf{f})$ and let $\mathbf{e}^{(i+1)}=\mathbf{u}-\mathbf{u}^{(i+1)}$ be the associated error. Then
\begin{equation}
\label{eq:bound_trunc1}
\lVert \mathbf{e}^{(i+1)}\rVert_A \leq {C_1}(\nu) \lVert \mathbf{e}^{(i)}\rVert_A,
\end{equation}
where, for small enough $\epsilon_\text{rel}$ and large enough $\nu$, $C_1(\nu)<1$ independent of the mesh size $h$.
\end{lemma}

\begin{proof}
For $s=1,\ldots,\nu$, let $\tilde{\mathbf{u}}^{(i)}_s$ be the quantity computed after application of the smoothing operator at step $s$ before truncation, and let $\mathbf{u}^{(i)}_s$ be the modification obtained from truncation by $\mathcal{T}_{\text{rel}_1}$ of \cref{eq:rel_tol1}. For example,
\begin{equation}
\tilde{\mathbf{u}}_1^{(i)} = \mathbf{u}^{(i)}+Q^{-1}(\mathbf{f}-A\mathbf{u}^{(i)}),\quad {\mathbf{u}}^{(i)}_1=\mathcal{T}_\text{rel$_1$}(\tilde{\mathbf{u}}^{(i)}_1).
\end{equation}
Denote the associated error as ${\mathbf{e}}^{(i)}_s=\mathbf{u}-{\mathbf{u}}^{(i)}_s$. From \cref{eq:rel_tol1}, we have
\begin{equation}
{\mathbf{e}}_1^{(i)} = (I-Q^{-1}A)\mathbf{e}^{(i)} + \delta_1^{(i)},\quad \text{where }\lVert \delta_1^{(i)}\rVert_2\leq \epsilon_\text{rel} \lVert \mathbf{r}^{(i)}\rVert_2.
\end{equation}
Similarly, after $\nu$ smoothing steps,
\begin{equation}
\label{eq:error_smooth}
\begin{aligned}
{\mathbf{e}}_\nu^{(i)} &= (I-Q^{-1}A)^\nu\mathbf{e}^{(i)} + \Delta_\nu^{(i)}\\
&=(I-Q^{-1}A)^\nu\mathbf{e}^{(i)} + (I-Q^{-1}A)^{\nu-1}\delta_1^{(i)} + \cdots + (I-Q^{-1}A)\delta_{\nu-1}^{(i)} + \delta_\nu^{(i)},
\end{aligned}
\end{equation}
where 
\begin{equation}
\label{eq:rel_tol_1}
\lVert \delta_s^{(i)}\rVert_2\leq \epsilon_\text{rel} \lVert \mathbf{r}^{(i)}\rVert_2,\quad s=1,\ldots,\nu.
\end{equation}
In Line \ref{lst:line:13} of \cref{alg:low-rank}, the residual $\tilde{\mathbf{r}}^{(i)}_\nu=A{\mathbf{e}}_\nu^{(i)}$ is truncated to ${\mathbf{r}}_\nu^{(i)}$ via \cref{eq:rel_tol2}, so that
\begin{equation}
\label{eq:rel_tol_2}
\lVert {\mathbf{r}}^{(i)}_\nu - \tilde{\mathbf{r}}^{(i)}_\nu \rVert_2 \leq \epsilon_\text{rel} h\lVert \mathbf{r}^{(i)}\rVert_2.
\end{equation}
Let $\tau^{(i)}={\mathbf{r}}^{(i)}_\nu - \tilde{\mathbf{r}}^{(i)}_\nu$. Referring to \cref{eq:error_exact,eq:error_smooth}, we can write the error associated with $\mathbf{u}^{(i+1)}$ as
\begin{equation}
\label{eq:error_trunc}
\begin{aligned}
\mathbf{e}^{(i+1)} & = {\mathbf{e}}_\nu^{(i)} - \mathscr{P}\bar{A}^{-1}\mathscr{R}{\mathbf{r}}_\nu^{(i)}\\
& = (I-\mathscr{P}\bar{A}^{-1}\mathscr{R}A) {\mathbf{e}}_\nu^{(i)} - \mathscr{P}\bar{A}^{-1}\mathscr{R} \tau^{(i)}\\
& = \mathbf{e}^{(i+1)}_\text{notrunc} + (A^{-1}-\mathscr{P}\bar{A}^{-1}\mathscr{R})A\Delta_\nu^{(i)}- \mathscr{P}\bar{A}^{-1}\mathscr{R} \tau^{(i)}\\
& = \mathbf{e}^{(i+1)}_\text{notrunc} + (A^{-1}-\mathscr{P}\bar{A}^{-1}\mathscr{R})(A\Delta_\nu^{(i)}+\tau^{(i)}) - A^{-1}\tau^{(i)}.
\end{aligned}
\end{equation}
Applying the approximation property \cref{eq:approx_property} gives
\begin{equation}
\label{eq:bound_1}
\lVert (A^{-1}-\mathscr{P}\bar{A}^{-1}\mathscr{R})(A\Delta_\nu^{(i)}+\tau^{(i)}) \rVert_A \leq C(\lVert A\Delta_\nu^{(i)}\rVert_2 +\lVert \tau^{(i)}\rVert_2).
\end{equation}
Using the fact that for any matrix $B\in\mathbb{R}^{N_xN_\xi\times N_xN_\xi}$,
\begin{equation}
\sup_{\mathbf{y}\neq\mathbf{0}} \frac{\lVert B\mathbf{y}\rVert_A}{\lVert\mathbf{y}\rVert_A}
= \sup_{\mathbf{y}\neq\mathbf{0}} \frac{\lVert A^{1/2}B\mathbf{y}\rVert_2}{\lVert A^{1/2}\mathbf{y}\rVert_2}
= \sup_{\mathbf{z}\neq\mathbf{0}} \frac{\lVert A^{1/2}BA^{-1/2}\mathbf{z}\rVert_2}{\lVert \mathbf{z}\rVert_2}
= \lVert A^{1/2}BA^{-1/2}\rVert_2,
\end{equation}
we get
\begin{equation}
\begin{aligned} \label{eq:a_delta}
\lVert A(I-Q^{-1}A)^{\nu-s}\delta_s^{(i)}\rVert_2 & \leq \lVert A^{1/2}\rVert_2\, \lVert (I-Q^{-1}A)^{\nu-s}\delta_s^{(i)}\rVert_A\\
& \leq \lVert A^{1/2}\rVert_2\, \lVert A^{1/2}(I-Q^{-1}A)^{\nu-s}A^{-1/2}\rVert_2\, \lVert\delta_s^{(i)}\rVert_A\\
& \leq \rho(I-Q^{-1}A)^{\nu-s}\lVert A^{1/2}\rVert_2^2\, \lVert\delta_s^{(i)}\rVert_2\\
\end{aligned}
\end{equation}
where $\rho$ is the spectral radius. We have used the fact that $A^{1/2}(I-Q^{-1}A)^{\nu-s}A^{-1/2}$ is a symmetric matrix (assuming $Q$ is symmetric). Define $d_1(\nu)=(\rho(I-Q^{-1}A)^{\nu-1}+\cdots+\rho(I-Q^{-1}A)+1) \lVert A^{1/2}\rVert_2^2$. Then \cref{eq:rel_tol_1,eq:rel_tol_2} imply that
\begin{equation}
\label{eq:bound_2}
\begin{aligned}
\lVert A\Delta_\nu^{(i)}\rVert_2 +\lVert \tau^{(i)}\rVert_2 & \leq \epsilon_\text{rel}(d_1(\nu)+h)\lVert \mathbf{r}^{(i)}\rVert_2\\
& \leq \epsilon_\text{rel}(d_1(\nu)+h)\lVert A^{1/2}\rVert_2\,\lVert \mathbf{e}^{(i)}\rVert_A.
\end{aligned}
\end{equation}
On the other hand, 
\begin{equation}
\label{eq:bound_3}
\begin{aligned}
\lVert A^{-1}\tau^{(i)} \rVert_A =(A^{-1}\tau^{(i)},\tau^{(i)})^{1/2} &\leq \lVert A^{-1}\rVert_2^{1/2}\,\lVert\tau^{(i)}\rVert_2 \\
& \leq \epsilon_\text{rel}h \lVert A^{-1}\rVert_2^{1/2}\, \lVert \mathbf{r}^{(i)}\rVert_2\\
&\leq \epsilon_\text{rel}h\lVert A^{-1}\rVert_2^{1/2}\,\lVert A^{1/2}\rVert_2\,\lVert \mathbf{e}^{(i)}\rVert_A.
\end{aligned}
\end{equation}
Combining \cref{eq:bound_exact,eq:error_trunc,eq:bound_1,eq:bound_2,eq:bound_3}, we conclude that
\begin{equation}
\lVert \mathbf{e}^{(i+1)}\rVert_A \leq {C_1}(\nu) \lVert \mathbf{e}^{(i)}\rVert_A
\end{equation}
where
\begin{equation}
{C_1}(\nu) = C\eta(\nu) +\epsilon_\text{rel} (C(d_1(\nu)+h)+h\lVert A^{-1}\rVert_2^{1/2})\lVert A^{1/2}\rVert_2.
\end{equation}
Note that $\rho(I-Q^{-1}A)<1$, $\lVert A\rVert_2$ is bounded by a constant, and $\lVert A^{-1}\rVert_2$ is of order $O(h^{-2})$ \cite{PoEl09}. Thus, for small enough $\epsilon_\text{rel}$ and large enough $\nu$, $C_1(\nu)$ is bounded below 1 independent of $h$.
\end{proof}

Next, we adjust this argument by considering the effect of the absolute truncations in the main \textbf{while} loop. In \cref{alg:low-rank}, the \textsc{Vcycle} is used for the residual equation, and the updated solution $\tilde{\mathbf{u}}^{(i+1)}$ and residual $\tilde{\mathbf{r}}^{(i+1)}$ are truncated to $\mathbf{u}^{(i+1)}$ and $\mathbf{r}^{(i+1)}$, respectively, using an absolute truncation criterion as in \cref{eq:abs_tol}. Thus,  at the $i$th iteration ($i>1$), the residual passed to the \textsc{Vcycle} function is in fact a perturbed residual, i.e.,
\begin{equation}
\mathbf{r}^{(i)}=\tilde{\mathbf{r}}^{(i)}+\beta=A\mathbf{e}^{(i)}+\beta,\quad \text{where } \lVert \beta \rVert_2\leq \sqrt{N_\xi}\epsilon_\text{abs}. 
\end{equation}
It follows that in the first smoothing step,
\begin{equation}
\tilde{\mathbf{u}}_1^{(i)} = \mathbf{u}^{(i)}+Q^{-1}(\mathbf{f}-A\mathbf{u}^{(i)}+\beta),\quad {\mathbf{u}}^{(i)}_1=\mathcal{T}_\text{rel$_1$}(\tilde{\mathbf{u}}^{(i)}_1),
\end{equation}
and this introduces an extra term in $\Delta_\nu^{(i)}$ (see \cref{eq:error_smooth}),
\begin{equation}
\Delta_\nu^{(i)}= (I-Q^{-1}A)^{\nu-1}\delta_1^{(i)} + \cdots + (I-Q^{-1}A)\delta_{\nu-1}^{(i)} + \delta_\nu^{(i)} -(I-Q^{-1}A)^{\nu-1}Q^{-1}\beta.
\end{equation}
As in the derivation of \cref{eq:a_delta}, we have
\begin{equation}
\lVert A(I-Q^{-1}A)^{\nu-1}Q^{-1}\beta \rVert_2 \leq \rho(I-Q^{-1}A)^{\nu-1}\lVert A^{1/2}\rVert_2^2\,\lVert Q^{-1}\rVert_2\,\lVert\beta\rVert_2.
\end{equation}
In the case of a damped Jacobi smoother (see \cref{eq:jacobi_smoother}), $\lVert Q^{-1}\rVert_2$ is bounded by a constant. Denote $d_2(\nu)=\rho(I-Q^{-1}A)^{\nu-1}\lVert A^{1/2}\rVert_2^2\,\lVert Q^{-1}\rVert_2$. Also note that $\lVert\mathbf{r}^{(i)}\rVert_2 \leq \lVert A^{1/2}\rVert_2\,\lVert\mathbf{e}^{(i)}\rVert_A+\lVert\beta\rVert$. Then \cref{eq:bound_2,eq:bound_3} are modified to
\begin{equation}  \label{eq:extra_1}
\begin{aligned}
& \lVert A\Delta_\nu^{(i)}\rVert_2 +\lVert \tau^{(i)}\rVert_2 \\
& \leq \epsilon_\text{rel}(d_1(\nu)+h)\lVert \mathbf{r}^{(i)}\rVert_2+d_2(\nu)\lVert\beta\rVert_2\\
& \leq \epsilon_\text{rel}(d_1(\nu)+h)\lVert A^{1/2}\rVert_2\,\lVert \mathbf{e}^{(i)}\rVert_A + (d_2(\nu)+ \epsilon_\text{rel}(d_1(\nu)+h))\lVert\beta\rVert_2,
\end{aligned}
\end{equation}
and
\begin{equation}  \label{eq:extra_2}
\lVert A^{-1}\tau^{(i)} \rVert_A \leq \epsilon_\text{rel}h\lVert A^{-1}\rVert_2^{1/2}\,\lVert A^{1/2}\rVert_2\,\lVert \mathbf{e}^{(i)}\rVert_A + \epsilon_\text{rel}h\lVert A^{-1}\rVert_2^{1/2}\,\lVert\beta\rVert.
\end{equation}
As we truncate the updated solution $\tilde{\mathbf{u}}^{(i+1)}$, we have
\begin{equation}  \label{eq:extra_3}
\mathbf{u}^{(i+1)}=\tilde{\mathbf{u}}^{(i+1)} + \gamma,\quad\text{where }\lVert\gamma\rVert_2\leq \sqrt{N_\xi}\epsilon_\text{abs}.
\end{equation}
Let 
\begin{equation}  \label{eq:extra_4}
C_2(\nu)=Cd_2(\nu)  +\epsilon_\text{rel} (C(d_1(\nu)+h)+h\lVert A^{-1}\rVert_2^{1/2})+\lVert A^{1/2}\rVert_2.
\end{equation}
From \cref{eq:extra_1,eq:extra_2,eq:extra_3,eq:extra_4}, we conclude with the following theorem:
\begin{theorem} \label{thm:convergence}
Let $\mathbf{e}^{(i)}=\mathbf{u}-\mathbf{u}^{(i)}$ denote the error at the $i$th iteration of \cref{alg:low-rank}. Then
\begin{equation} \label{eq:bound_trunc2}
\lVert \mathbf{e}^{(i+1)}\rVert_A \leq {C_1}(\nu)\lVert \mathbf{e}^{(i)}\rVert_A + C_2(\nu)\sqrt{N_\xi}\epsilon_\text{abs},
\end{equation}
where $C_1(\nu)<1$ for large enough $\nu$ and small enough $\epsilon_\text{rel}$, and $C_2(\nu)$ is bounded by a constant. Also, \cref{eq:bound_trunc2} implies that
\begin{equation}
\lVert \mathbf{e}^{(i)}\rVert_A \leq C_1^i(\nu) \lVert \mathbf{e}^{(0)}\rVert_A + \frac{1-C_1^{i}(\nu)}{1-C_1(\nu)} C_2(\nu)\sqrt{N_\xi}\epsilon_\text{abs},
\end{equation}
i.e., the $A$-norm of the error for the low-rank multigrid solution at the $i$th iteration is bounded by $C_1^i(\nu) \lVert \mathbf{e}^{(0)}\rVert_A+O(\sqrt{N_\xi}\epsilon_\text{abs})$. Thus, \cref{alg:low-rank} converges until the $A$-norm of the error becomes as small as $O(\sqrt{N_\xi}\epsilon_\text{abs})$.
\end{theorem}

It can be shown that the same result holds if post-smoothing is used. Also, the convergence of full (recursive) multigrid with these truncation operations can be established following an inductive argument analogous to that in the deterministic case (see, e.g., \cite{ElSi14,Hack85}). Besides, in \cref{alg:low-rank}, the truncation on $\tilde{\mathbf{r}}^{(i+1)}$ imposes a stopping criterion, i.e.,
\begin{equation} \label{eq:stopping}
\begin{aligned}
\lVert \tilde{\mathbf{r}}^{(i+1)} \rVert_2 & \leq \lVert \tilde{\mathbf{r}}^{(i+1)}-\mathbf{r}^{(i+1)}\rVert_2 + \lVert {\mathbf{r}}^{(i+1)} \rVert_2\\
& \leq \sqrt{N_\xi}\epsilon_\text{abs} + tol*r_0.
\end{aligned}
\end{equation} 
In \cref{sec:numerical} we will vary the value of $\epsilon_\text{abs}$ and see how the low-rank multigrid solver works compared with \cref{alg:mg} where no truncation is done.

\begin{remark}
It is shown in \cite{PoEl09} that for \cref{eq:galerkin}, with constant mean $c_0$ and standard deviation $\sigma$,
\begin{equation}
\lVert A\rVert_2 = \alpha(c_0 + \sigma C_{p+1}^\text{max} \sum_{l=1}^m \sqrt{\lambda_l} \lVert c_l(x)\rVert_\infty ),
\end{equation}
where $C_{p+1}^\text{max}$ is the maximal root of an orthogonal polynomial of degree $p+1$, and $\alpha$ is a constant independent of $h$, $m$, and $p$. If Legendre polynomials on the interval $[-1,1]$ are used, $C_{p+1}^\text{max}<1$. Since both $C_1$ and $C_2$ in \cref{thm:convergence} are related to $\lVert A\rVert_2$, the convergence rate of \cref{alg:low-rank} will depend on $m$. However, if the eigenvalues $\{\lambda_l\}$ decay fast, this dependence is negligable.
\end{remark}

\begin{remark}
If instead a relative truncation is used in the \textbf{while} loop so that
\begin{equation}
\mathbf{r}^{(i+1)}=\tilde{\mathbf{r}}^{(i+1)}+\beta=A\mathbf{e}^{(i+1)}+\beta,\quad \text{where } \lVert \beta \rVert_2\leq \epsilon_\text{rel}\lVert\tilde{\mathbf{r}}^{(i+1)}\rVert,
\end{equation}
then a similar convergence result can be derived, and the algorithm stops when
\begin{equation}
\lVert \tilde{\mathbf{r}}^{(i+1)} \rVert_2  \leq \frac{tol*r}{1-\epsilon_\text{rel}}.
\end{equation}
However, the relative truncation in general results in a larger rank for $\mathbf{r}^{(i)}$, and the improvement in efficiency will be less significant. 
\end{remark}
 
\section{Numerical experiments}
\label{sec:numerical}
Consider the benchmark problem with a two-dimensional spatial domain $D=(-1,1)^2$ and constant source term $f=1$. We look at two different forms for the covariance function $r(x,y)$ of the diffusion coefficient $c(x,\omega).$

\subsection{Exponential covariance}
\label{sec:numerical1}
The exponential covariance function takes the form
\begin{equation} \label{eq:cov}
r(x,y)=\sigma^2\text{exp}\left(-\frac{1}{b} \lVert x-y\rVert_1\right).
\end{equation}
This is a convenient choice because there are known analytic solutions for the eigenpair ($\lambda_l$,$c_l(x)$) \cite{GhSp03}. In the KL expansion, take $c_0(x)=1$ and $\{\xi_l\}_{l=1}^m$ independent and uniformly  distributed on $[-1,1]$:
\begin{equation}
\label{eq:ex_kl}
c(x,\omega) = c_0(x) + \sqrt{3}\sum_{l=1}^m \sqrt{\lambda_l}c_l(x)\xi_l(\omega).
\end{equation}
Then $\sqrt{3}\xi_l$ has zero mean and unit variance, and Legendre polynomials are used as basis functions for the stochastic space. The correlation length $b$ affects the decay of $\{\lambda_l\}$ in the KL expansion. The number of random variables $m$ is chosen so that 
\begin{equation} \label{eq:lambda}
\left(\sum_{l=1}^m\lambda_l\right)\Big/\left(\sum_{l=1}^M\lambda_l\right)\geq 95\%.
\end{equation}
Here $M$ is a large number which we set as 1000. 

We now examine the performance of the multigrid solver with low-rank truncation. We employ a damped Jacobi smoother, with 
\begin{equation} \label{eq:jacobi_smoother}
Q=\frac{1}{\omega}D,\quad D=\text{diag}(A)=I\otimes\text{diag}(K_0),
\end{equation}
and apply three smoothing steps ($\nu=3$) in the \textsc{Smooth} function. Set the multigrid $tol=10^{-6}$. As shown in \cref{eq:stopping}, the relative residual $\lVert F-\mathcal{A}(U^{(i)})\rVert_F/\lVert F\rVert_F$ for the solution $U^{(i)}$ produced in \cref{alg:low-rank} is related to the value of the truncation tolerance $\epsilon_\text{abs}$. In all the experiments, we also run the multigrid solver without truncation to reach a relative residual that is closest to what we get from the low-rank multigrid solver. We fix the relative truncation tolerance $\epsilon_\text{rel}$ as $10^{-2}$. (The truncation criteria in \cref{eq:rel_tol1,eq:rel_tol2} are needed for the analysis. In practice we found the performance with the relative criterion in \cref{eq:rel_tol_equiv} to be essentially the same as the results shown in this section.) The numerical results, i.e., the rank of multigrid solution, the number of iterations, and the elapsed time (in seconds) for solving the Galerkin system, are given in \cref{table:n,table:m,table:sigma}. In all the tables, the 3rd and 4th columns are the results of low-rank multigrid with different values of truncation tolerance $\epsilon_\text{abs}$, and for comparison the last two columns show the results for the multigrid solver without truncation. The Galerkin systems are generated from the Incompressible Flow and Iterative Solver Software (IFISS, \cite{ifiss}). All computations are done in MATLAB 9.1.0 (R2016b) on a MacBook with 1.6 GHz Intel Core i5 and 4 GB SDRAM.

\Cref{table:n} shows the performance of the multigrid solver for various mesh sizes $h$, or spatial degrees of freedom $N_x$, with other parameters fixed. The 3rd and 5th columns show that multigrid with low-rank truncation uses less time than the standard multigrid solver. This is especially true when $N_x$ is large: for $h=2^{-8}$, $N_x=261121$, low-rank approximation reduces the computing time from 2857s to 370s.  The improvement is much more significant (see the 4th and 6th columns) if the problem does not require very high accuracy for the solution. \Cref{table:m} shows the results for various degrees of freedom $N_\xi$ in the stochastic space. The multigrid solver with absolute truncation tolerance $10^{-6}$ is more efficient compared with no truncation in all cases and uses only about half the time. The 4th and 6th columns indicate that the decrease in computing time by low-rank truncation is more obvious with the larger tolerance $10^{-4}$.

\begin{table}[tbhp]
\caption{Performance of multigrid solver with $\epsilon_\text{abs}=10^{-6}$, $10^{-4}$, and no truncation for various $N_x=(2/h-1)^2$. Exponential covariance, $\sigma=0.01$, $b=4$, $m=11$, $p=3$, $N_\xi=364$.}
\label{table:n}
\centering
\begin{tabular}{|l|l|c|c|c|c|}
\hline
\multicolumn{2}{|c|}{} & $\epsilon_\text{abs}=10^{-6}$ & $\epsilon_\text{abs}=10^{-4}$ & \multicolumn{2}{c|}{No truncation}\\
\hline
\multirow{4}{2.3cm}{$64\times 64$ grid  $h=2^{-5}$  $N_x=3969$} 
& Rank & 51 & 12 & & \\
 & Iterations & 5 & 4 & 5 & 4 \\
 & Elapsed time & 6.26	&	1.63	&	12.60	&	10.08\\
& Rel residual & 1.51e-6	&	6.05e-5	&	9.97e-7	&	1.38e-5\\
\hline
\multirow{4}{2.3cm}{$128\times 128$ grid  $h=2^{-6}$  $N_x=16129$} 
& Rank & 51 & 12 & & \\
 & Iterations & 6 & 4 & 5 &3 \\
 & Elapsed time & 20.90	&	5.17	&	54.59	&	32.92\\
& Rel residual & 2.45e-6	&	9.85e-5	&	1.23e-6	&	2.20e-4\\
\hline
\multirow{4}{2.3cm}{$256\times 256$ grid  $h=2^{-7}$  $N_x=65025$} 
 & Rank & 49 & 13 & & \\
 & Iterations & 5 & 4 & 5 & 3 \\
 & Elapsed time & 76.56	&	24.31	&	311.27	&	188.70\\
& Rel residual & 4.47e-6	&	2.07e-4	&	1.36e-6	&	2.35e-04\\
\hline
\multirow{4}{2.3cm}{$512\times 512$ grid  $h=2^{-8}$  $N_x=261121$ } 
& Rank & 39 & 16 & & \\
 & Iterations & 5 & 3 & 4 & 3 \\
& Elapsed time & 370.98	&	86.30	&	2857.82	&	2099.06\\
& Rel residual & 9.93e-6	&	4.33e-4	&	1.85e-5	&	2.43e-4\\
\hline
\end{tabular}
\end{table}

\begin{table}[tbhp]
\caption{Performance of multigrid solver with $\epsilon_\text{abs}=10^{-6}$, $10^{-4}$, and no truncation for various $N_\xi=(m+p)!/(m!p!)$. Exponential covariance, $\sigma=0.01$, $h=2^{-6}$, $p=3$, $N_x=16129$.}
\label{table:m}
\centering
\begin{tabular}{|l|l|c|c|c|c|}
\hline
\multicolumn{2}{|c|}{} & $\epsilon_\text{abs}=10^{-6}$ & $\epsilon_\text{abs}=10^{-4}$ & \multicolumn{2}{c|}{No truncation}\\
\hline
& Rank & 25 & 9 & & \\
$b=5,m=8$ & Iterations & 5 & 4 & 5 &3 \\
$N_\xi=165$ & Elapsed time & 5.82	&	1.71	&	19.33	&	11.65\\
& Rel residual & 5.06e-6	&	3.41e-4	&	1.22e-6	&	2.20e-4\\
\hline
 & Rank & 51 & 12 & & \\
$b=4,m=11$ & Iterations & 6 & 4 & 5 & 3 \\
$N_\xi=364$ & Elapsed time & 20.90	&	5.17	&	54.59	&	32.92\\
& Rel residual & 2.45e-6	&	9.85e-5	&	1.23e-6	&	2.20e-4\\
\hline
& Rank & 91 &23 & & \\
$b=3,m=16$ & Iterations & 6 & 5 & 5 & 4 \\
$N_\xi=969$ & Elapsed time & {97.34}	&	16.96	&	197.82	&	158.56\\
& Rel residual & 5.71e-7	&	3.99e-5	&	1.23e-6	&	1.63e-5\\
\hline
& Rank & 165 & 86 & & \\
$b=2.5,m=22$ & Iterations & 6 & 5 & 6 & 4 \\
$N_\xi=2300$ & Elapsed time &{648.59}	&	172.41	&	1033.29	&	682.45\\
& Rel residual & 1.59e-7	&	8.57e-6	&	9.29e-8	&	1.63e-5\\
\hline
\end{tabular}
\end{table}

We have observed that when the standard deviation $\sigma$ in the covariance function (\ref{eq:cov}) is smaller, the singular values of the solution matrix $U$ decay faster (see \cref{fig:decay1}), and it is more suitable for low-rank approximation. This is also shown in the numerical results. In the previous cases, we fixed $\sigma$ as 0.01. In \cref{table:sigma}, the advantage of low-rank multigrid is clearer for a smaller $\sigma$, and the solution is well approximated by a matrix of smaller rank. On the other hand, as the value of $\sigma$ increases, the singular values of the matricized solution, as well as the matricized iterates, decay more slowly and the same truncation criterion gives higher-rank objects. Thus, the total time for solving the system and the time spent on truncation will also increase. Another observation from the above numerical experiments is that the iteration counts are largely unaffected by truncation. In \cref{alg:low-rank}, similar numbers of iterations are required to reach a comparable accuracy as in the cases with no truncation.

\begin{table}[tbph] 
\caption{Performance of multigrid solver with $\epsilon_\text{abs}=10^{-6}$, $10^{-4}$, and no truncation for various $\sigma$. Time spent on truncation is given in parentheses. Exponential covariance, $b=4$, $h=2^{-6}$, $m=11$, $p=3$, $N_x=16129$, $N_\xi=364.$}
\label{table:sigma}
\centering
\begin{tabular}{|l|l|c|c|c|c|}
\hline
\multicolumn{2}{|c|}{} & $\epsilon_\text{abs}=10^{-6}$ & $\epsilon_\text{abs}=10^{-4}$ & \multicolumn{2}{c|}{No truncation}\\
\hline
\multirow{4}{2cm}{$\sigma=0.001$} & Rank & 13 & 12 & & \\
& Iterations & 6 & 4 & 5 & 4 \\
& Elapsed time & 7.61 (4.77)	&	3.73 (2.29)	&	54.43	&	43.58\\
& Rel residual & 1.09e-6	&	6.53e-5	&	1.22e-6	&	1.63e-5\\
\hline
\multirow{4}{2cm}{$\sigma=0.01$} & Rank & 51 & 12 & & \\
& Iterations & 6 & 4 & 5 &3 \\
& Elapsed time & 20.90 (15.05)	&	5.17 (3.16)	&	54.59	&	32.92\\
& Rel residual & 2.45e-6	&	9.85e-5	&	1.23e-6	&	2.20e-4\\
\hline
\multirow{4}{2cm}{$\sigma=0.1$} & Rank & 136 & 54 & & \\
& Iterations & 6 & 4 & 5 &3 \\
& Elapsed time & {54.44 (33.91)}	&	{18.12 (12.70)}	&	55.49	&	33.62\\
& Rel residual & 3.28e-6	&	2.47e-4	&	1.88e-6	&	2.62e-4\\
\hline
\multirow{4}{2cm}{$\sigma=0.3$} & Rank & 234 & 128 & & \\
& Iterations & 9 & 7 & 8 &4 \\
& Elapsed time & {138.63 (77.54)}	&	{60.96 (38.66)}	&	86.77	&	43.42\\
& Rel residual & 6.03e-6	&	4.71e-4	&	2.99e-6	&	7.76e-4\\
\hline
\end{tabular}
\end{table}

\subsection{Squared exponential covariance}
\label{sec:numerical2}
In the second example we consider covariance function
\begin{equation} \label{eq:cov2}
r(x,y)=\sigma^2\text{exp}\left(-\frac{1}{b^2} \lVert x-y\rVert_2^2\right).
\end{equation}
The eigenpair $(\lambda_l,c_l(x))$ is computed via a Galerkin approximation of the eigenvalue problem
\begin{equation} 
\int_D r(x,y) c_l(y)\text{d}y=\lambda_l c_l(x).
\end{equation}
Again, in the KL expansion \cref{eq:ex_kl}, take $c_0(x)=1$ and $\{\xi_l\}_{l=1}^m$ independent and uniformly  distributed on $[-1,1]$. The eigenvalues of the squared exponential covariance \cref{eq:cov2} decay much faster than those of \cref{eq:cov}, and thus fewer terms are required to satisfy \cref{eq:lambda}. For instance, for $b=2$, $m=3$ will suffice. \Cref{table:n2} shows the performance of multigrid with low-rank truncation for various spatial degrees of freedom $N_x$. In this case, we are able to work with finer meshes since the value of $N_\xi$ is smaller. In all experiments the low-rank multigrid solver uses less time compared with no truncation.

\begin{table}[tbhp]
\caption{Performance of multigrid solver with $\epsilon_\text{abs}=10^{-6}$, $10^{-4}$, and no truncation for various $N_x=(2/h-1)^2$. Squared exponential covariance, $\sigma=0.01$, $b=2$, $m=3$, $p=3$, $N_\xi=20.$}
\label{table:n2}
\centering
\begin{tabular}{|l|l|c|c|c|c|}
\hline
\multicolumn{2}{|c|}{} & $\epsilon_\text{abs}=10^{-6}$ & $\epsilon_\text{abs}=10^{-4}$ & \multicolumn{2}{c|}{No truncation}\\
\hline
\multirow{4}{2.5cm}{$128\times 128$ grid  $h=2^{-6}$  $N_x=16129$} 
& Rank & 9 & 4 & & \\
& Iterations & 5 & 3 & 4 &3 \\
& Elapsed time & 0.78	&	0.35	&	1.08	&	0.82\\
& Rel residual & 1.20e-5	&	9.15e-4	&	1.63e-5	&	2.20e-4\\
\hline
\multirow{4}{2.5cm}{$256\times 256$ grid  $h=2^{-7}$  $N_x=65025$} 
& Rank & 8 & 4 & & \\
& Iterations & 4 & 3 & 4 & 3 \\
& Elapsed time & 2.55	&	1.31	&	4.58	&	3.46\\
& Rel residual & 3.99e-5	&	9.09e-4	&	1.78e-05	& 2.35e-4\\
\hline
\multirow{4}{2.5cm}{$512\times 512$ grid  $h=2^{-8}$  $N_x=261121$ } 
 & Rank & 8 & 2 & & \\
& Iterations & 4 & 2 & 4 & 2 \\
& Elapsed time & 10.23	&	2.13	&	18.93	&	9.61\\
& Rel residual & 6.41e-5	&	6.91e-3	&	1.85e-5	&	3.29e-3\\
\hline
\multirow{4}{2.5cm}{$1024\times 1024$ grid \hspace{0.5cm}      $h=2^{-9}$  $N_x=1045629$ } 
& Rank & 8 & 2 & & \\
& Iterations & 4 & 2 & 4 & 2 \\
& Elapsed time & 58.09	&	10.66	&	115.75	&	63.54\\
& Rel residual & 6.41e-5	&	6.93e-3	&	1.90e-5	&	3.32e-3\\
\hline
\end{tabular}
\end{table}

\section{Conclusions}
In this work we focused on the multigrid solver, one of the most efficient iterative solvers, for the stochastic steady-state diffusion problem. We discussed how to combine the idea of low-rank approximation with multigrid to reduce computational costs. We proved the convergence of the low-rank multigrid method with an analytic error bound. It was shown in numerical experiments that the low-rank truncation is useful in decreasing the computing time when the variance of the random coefficient is relatively small. The proposed algorithm also exhibited great advantage for problems with large number of spatial degrees of freedom.

\bibliographystyle{siamplain}
\bibliography{references}

\end{document}